\documentclass[11pt, oneside]{article}

\usepackage{amssymb,amsmath,amsthm}
\usepackage{graphicx}                   
\usepackage{mathrsfs}
\usepackage{hyperref}

\newcommand{\arcsinh}{\operatorname{arcsinh}}

\newcommand{\R}{\mathbb R}
\newcommand{\N}{\mathbb N}

\newcommand{\ep}{\varepsilon}
\newcommand{\ga}{\gamma}
\renewcommand{\d}{\; \mathrm{d}}

\newcommand{\abs}[1]{\left\lvert#1\right\rvert}
\newcommand{\norm}[1]{\left\|#1\right\|}
\newcommand{\set}[1]{\left\{#1\right\}}
\newcommand{\floor}[1]{\left\lfloor #1 \right\rfloor}

\renewcommand{\dim}{m}

\newcommand{\q}{q}
\newcommand{\act}{\mathscr{L}}
\newcommand{\jac}{\mathscr{M}}
\newcommand{\nac}{\mathscr{N}}

\newtheorem{theorem}{Theorem}[section]
\newtheorem{lemma}[theorem]{Lemma}
\newtheorem{corollary}[theorem]{Corollary}
\newtheorem*{remark*}{Remark}
\newtheorem*{definition}{Definition}

\title{Effective Hamiltonian Dynamics via the Maupertuis Principle}
\author{Hartmut Schwetlick, Daniel C. Sutton and Johannes Zimmer}

\date{Dedicated to J\"urgen Scheurle, who introduced the first and third authors to dynamical systems, on the occasion of his 65th
  birthday}
\begin{document}
	
\maketitle
	
\begin{abstract}
  We consider the dynamics of a Hamiltonian particle forced by a rapidly oscillating potential in $\dim$-dimensional space. As
  alternative to the established approach of averaging Hamiltonian dynamics by reformulating the system as Hamilton-Jacobi
  equation, we propose an averaging technique via reformulation using the Maupertuis principle. We analyse the result of these two
  approaches for one space dimension.  For the initial value problem the solutions converge uniformly when the total energy is
  fixed. If the initial velocity is fixed independently of the microscopic scale, then the limit solution depends on the choice of
  subsequence. We show similar results hold for the one-dimensional boundary value problem. In the higher dimensional case we show
  a novel connection between the Hamilton-Jacobi and Maupertuis approaches, namely that the sets of minimisers and saddle points
  coincide for these functionals.
\end{abstract}
	
\section{Introduction}
	
Hamiltonian systems with rapidly oscillating potentials appear in numerous applications, and the description of their effective
(averaged) behaviour is of great interest. The seminal work of Lions, Papanicolaou and Varadhan~\cite{Lions1987a} derives this
description via a reformulation of the problem as a Hamilton-Jacobi equation, which is then homogenised. Important contributions
in this direction are also due to E~\cite{E1991a} and Evans and Gomes~\cite{Evans2001a,Evans2002a}, to name a few. In this note,
we suggest an alternative approach, via an equivalent reformulation of the original Hamiltonian problem as a least action problem,
which is a geodesic formulation in the so-called Jacobi metric. This formulation is often attributed to Maupertuis, though it has
been studied before by many scientists, including Fermat and Bernoulli~\cite{Biesiada1995a}; see also~\cite{Jourdain1912a}.

As a model problem, we study the effective description of the motion of a particle in $\dim$-dimensional space when subjected to a
highly oscillatory potential energy. This is modelled as
\begin{equation}
  \label{n2l-4}
  \q''(t) = -\frac{1}{\ep}\nabla V\left({\q(t)}/{\ep}\right) \text{ for }  \q \colon [0, T] \to \mathbb R^\dim,
\end{equation}
where $\ep > 0$ and the potential energy function $V \in C^2(\R^\dim)$ is $1$-periodic in each spatial direction. The mass of the
particle is taken to be $1$. We determine the effective behaviour of solutions to~\eqref{n2l-4} as $\ep \to 0$. In this case, it
is well known that~\eqref{n2l-4} is the Euler-Lagrange equation for the \emph{action functional}
\begin{equation}
  \label{action}
  \act_T^{\ep}[q] := \int_0^T \left[ \frac{1}{2}\norm{\q'(t)}^2 -V\left({\q(t)}/{\ep}\right) \right] \d t.
\end{equation}
Traditionally the dependence on $T$ is implicit. However, here it will be useful to denote it explicitly. Hamilton's principle of
least action states that the evolution of a mechanical system is a critical point of the action functional.
	
One way to obtain information about the limiting dynamics of~\eqref{n2l-4} as $\ep \to 0$ is to observe that~\eqref{n2l-4} gives
rise to the Hamiltonian
\begin{equation*}
  \mathscr H_{\ep}(p,q) = \frac{1}{2}\norm{p}^2 + V\left(\frac{q}{\ep}\right),
\end{equation*}
where $p$ and $q$ are the generalised momenta and position, respectively. An account of how~\eqref{n2l-4} and the Hamilton-Jacobi
PDE are related is included in~\cite[Chapter 3]{Sutton2013a}. The study of the associated Hamilton-Jacobi PDE and its effective
description is described in a more general context in the seminal work~\cite{Lions1987a}. Our comparison here is for a particular
Hamiltonian.
	
In this paper we propose an alternative framework to determine the effective dynamics of~\eqref{n2l-4} by studying this equation
in the context of the Maupertuis principle~\cite[p. 243]{Arnold1989a}. This principle states that solutions of~\eqref{n2l-4} are
reparametrised stationary points of the \emph{Jacobi functional}
\begin{equation}
  \label{length}
  \jac_E^{\ep}[\ga] := \int_0^1 2(E-V(\ga(s)/\ep)) \norm{\ga'(s)}^2 \d s \text{ for } \ga \colon [0,1] \to \mathbb R^\dim,
\end{equation}
where $E \in \mathbb R$ represents the total energy of the mechanical system. This reparametrisation from the geometric
parametrisation by $s$ to physical time $t$ is given by
\begin{equation}
  \label{eq:maupert-time}
  t_{\ep}(s) = \int_0^s \frac{\|\ga'(\sigma)\|}{\sqrt{2(E-V(\ga(\sigma)/\ep))}} \d \sigma.
\end{equation}
Observe that if $V$ is bounded and $E > \norm{V}_{\infty}$, then~\eqref{length} corresponds to the energy functional for curves in
a Riemannian manifold. A proof of the correspondence between stationary points, for sufficiently smooth potentials, can be found
in~\cite{Arnold1989a, Marsden1999a}. The proof typically involves showing that
\begin{equation}
  \label{eq:le}
  \nac_E^{\ep}[\ga] := \int_0^1 \sqrt{2(E-V(\ga(s)/\ep))} \norm{\ga'(s)} \d s
\end{equation}
and~\eqref{action} share the same critical points via their Euler-Lagrange equation. The fact that~\eqref{length} and
\eqref{eq:le} share the same minimisers when parametrised by arc-length is well known~\cite[Lemma 1.4.2, 1.4.5]{Jost2005a};
underlying is the observation that $\nac_E^{\ep}$ is a length functional, while $\act_T^{\ep}$ is an energy functional.

Both the Maupertuis averaging we propose here and the established homogenisation procedure via Hamilton-Jacobi theory are viable
homogenisation routes. A priori, it is not clear they agree. To answer this question, we study the initial value problem and
boundary value problem for~\eqref{n2l-4} in one space dimension. While the computations are relatively straightforward, it turns
out that the problem has to be formulated carefully. In principle, two natural options exist for the initial value
problem. Namely, in addition to the initial position, one can either fix the initial velocity or the total energy $E$. For the
boundary value problem, one can analogously choose the initial position and then either fix the arrival time or the total energy
$E$. As described in Section~\ref{sec:effdyn}, only the problems with fixed energy lead to a unique solution in the limit
$\ep\to0$.

In Section~\ref{sec:Effect-one-dimension}, we show the explicit averaging results are reproduced both by Hamilton-Jacobi theory
and the new approach via Maupertuis advocated here. It is worth pointing out that the Maupertuis approach automatically keeps the
energy $E$ fixed, and is thus well suited for the cases identified as relevant in Section~\ref{sec:effdyn}. Furthermore, the
Maupertuis result turns out to be trivial in one space dimension, and agrees with the Hamilton-Jacobi result, which also
reproduces the explicitly computed limit, albeit in a less trivial manner than via the Maupertuis route.

Motivated by the positive results obtained for one space dimension, it is natural to consider Maupertuis averaging in higher
dimensions. In this article, we provide only the first steps by collecting general convergence results. It is also of interest to
understand if minimisers for the Maupertuis principle are minimisers of the Hamiltonian action functional, and if the same holds
true for saddle points. While the result that their critical points agree under natural assumptions is a classic one (see,
e.g.,~\cite{Arnold1989a, Marsden1999a}), we could not find a breakdown of this result for minima and saddle points (neither of
these functionals have, in a natural setting, maxima), and establish this result in Section~\ref{sec:Relat-betw-crit}.

\section{Effective one-dimensional dynamics: Explicit computation}
\label{sec:effdyn}
	
In this section, we restrict the analysis to the one-dimensional case and consider the motion of a particle with unit mass
travelling in the rapidly oscillating potential. We model this with a potential of the form $V(\cdot/\ep)$, where $\ep$ is a
positive parameter and $V \in C^{2}(\R)$ is $1-$periodic.  Without loss of generality, we assume that $\left[ 0, 1\right]$ is the
periodicity cell and that $\max V = 0$. This case is covered by~\cite{Lions1987a}, as discussed in
Section~\ref{sec:Effect-one-dimension-2}. The motion of the particle is then governed by equation
\begin{equation}
  \label{n2l} 
  \q''_{\ep}(t) = -\frac{1}{\ep} \,  V'\left(\q_{\ep}(t)/\ep\right),\quad \; t \in \R.
\end{equation}
Consider a sequence $\{q_{\ep}\}_{\ep>0}$, where $q_{\ep} \colon \R \to \R$ solves~\eqref{n2l} and $\q_{\ep}(0)=\q_a$. Then the
total energy
\begin{equation}
  \label{1d-energy}
  E_{\ep}(t) := \frac{1}{2}\abs{\q'_{\ep}(t)}^2 + V\left( \q_{\ep}(t)/\ep \right) > 0
\end{equation}
is conserved along trajectories, and we denote this positive constant $E_{\ep}$; i.e., $E_{\ep}(t) \equiv E_{\ep} > 0$ for all
$t \in \R$. Our aim is to study the convergence of $\q_{\ep}$ as $\ep \to 0$. In order to achieve this, we solve~\eqref{1d-energy}
to obtain
\begin{equation}
  \label{eq:qprime}
  \q_{\ep}'(t) = \pm \sqrt{2(E_\ep-V(\q_{\ep}(t)/\ep))}.
\end{equation}
In this paper we will, without loss of generality, take $\q_{\ep}' > 0$; this corresponds to an initial value problem with
positive speed or a boundary value problem connecting $q_a$ to $q_b$ with $q_a < q_b$. Since $q_\ep$ is strictly monotone, there
exists an inverse function $t_\ep$ for $q_{\ep}$ satisfying $t_{\ep}(\q_{\ep}(t)) = t$ for $t\in\R$. Then from~\eqref{eq:qprime}
one obtains
\begin{equation}
  \label{eq:time-inverse}
  t_{\ep}(q) = \int_{q_a}^q \frac{1}{\sqrt{2\left(E_{\ep}-V\left({\ga}/{\ep}\right)\right)}} \d \ga . 
\end{equation}
For notational convenience, we define
\begin{equation}
  \label{def:sigma-E}
  \sigma(E_{\ep}) := \int_0^1 \frac{1}{\sqrt{2(E_{\ep}-V(\ga))}} \d \ga. 
\end{equation}

\subsection{The one-dimensional initial value problem}
\label{sec:ivp-1d}

In this section, we solve~\eqref{n2l} with $E_\ep >0$ subject to the condition that $q_{\ep}(0) = q_a$ and
$q_{\ep}'(0) = p_{\ep}$. Since $\max V = 0$, it holds that $E_\ep > V$. We consider two cases --- one where $p_{\ep}$ is kept
constant and one where $p_{\ep}$ is chosen to ensure that $E_\ep$ is constant. We first show that the sequence $t_{\ep}$ defined
by~\eqref{eq:time-inverse} approaches a line uniformly as $\ep \to 0$.

\begin{lemma}
  \label{tbd}
  The sequence $t_{\ep}$ defined in~\eqref{eq:time-inverse} satisfies
  \begin{equation}
    \label{test}
    \abs{t_{\ep}(q) - \sigma(E_{\ep}) (q-q_a)  } \leq {C_{E_{\ep}}}\ep,
  \end{equation}
  where $\sigma(E_{\ep})$ is defined in~\eqref{def:sigma-E} and
  \begin{equation}
    \label{const}
    C_{E_{\ep}}:=\frac{2}{\sqrt{2E_{\ep}}} . 
  \end{equation}
\end{lemma}

\begin{proof}
  Let $\ep > 0$ and $q \in \R$. It follows that
  \begin{align*}
    \abs{t_{\ep}(q) - \sigma(E_{\ep}) (q-q_a) } 
    &= \abs{\int_{q_a}^{q} \left[ \frac{1}{\sqrt{2(E_{\ep}-V(\ga/\ep))}} 
      - \sigma(E_{\ep}) \right] \d \ga }\\ 
    &= \abs{\ep \int_{q_a/\ep}^{q/\ep} \left[ \frac{1}{\sqrt{2(E_{\ep}-V(\tilde \ga))}} 
      - \sigma(E_{\ep}) \right] \d \tilde \ga }.
  \end{align*}
  If $\ep$ is such that $q - q_a > \ep $, then the interval $\left[{q_a}/\ep,q/\ep\right]$ is the union of
  $\floor{\frac{q - q_a}{\ep}}$ fundamental periods of $V$ (where $\floor{\cdot}$ denotes the floor function) and a remainder set
  $Q_{\ep}$ of measure $\text{meas}(Q_{\ep}) =\frac{q - q_a}{\ep} -\floor{\frac{q - q_a}{\ep}} \leq 1$. Hence
  \begin{align}
    \label{t-est-large}
    &\abs{\ep \int_{q_a/\ep}^{q/\ep} \left[ \frac{1}{\sqrt{2(E_{\ep}-V(\tilde\ga))}} - \sigma(E_{\ep}) \right] \d \tilde\ga } 
      \nonumber \\
    &= \abs{\ep \int_{Q_{\ep}} \left[ \frac{1}{\sqrt{2(E_{\ep}-V(\tilde\ga))}} - \sigma(E_{\ep}) \right] \d \tilde\ga } \nonumber \\ 
    &\leq \ep  \int_{Q_{\ep}} \abs{ \frac{1}{\sqrt{2(E_{\ep}-V(\tilde\ga))}} - \sigma(E_{\ep}) } \d \tilde\ga  
    \leq \frac{2\ep}{\sqrt{2E_{\ep}}} . 
  \end{align}
  Otherwise $q - q_a \leq \ep $ and therefore
  \begin{multline}
    \label{t-est-small}
    \abs{\int_{q_a}^{q} \left[ \frac{1}{\sqrt{2\left(E_{\ep}-V\left(\ga/{\ep}\right)\right)}} 
      	- \sigma(E_{\ep}) \right] \d \ga }\\ 
    \leq  \int_{q_a}^{q} \left| \frac{1}{\sqrt{2\left(E_{\ep}-V\left(\ga/{\ep}\right)\right)}}
      + \sigma(E_{\ep}) \right|  \d \ga   \leq  \frac{2\ep}{\sqrt{2E_{\ep}}} . 
  \end{multline}%% JZ
  Combining~\eqref{t-est-large} and~\eqref{t-est-small}, we obtain~\eqref{test}.
\end{proof}
	
We now show that the solutions of the initial value problem for equation~\eqref{n2l} converge uniformly, as $\ep \to 0$, if the
initial speed $p_{\ep}$ varies such that the associated energy is constant, $E := E_\ep$ for all $\ep$.
	
\begin{theorem}
  \label{mthm}
  Fix $E > 0$. Let $\q_\ep$ solve~\eqref{n2l} subject to $\q(0) = \q_a$, with $\q'(0) = p_{\ep} > 0$ chosen such
  that~\eqref{1d-energy} is satisfied with constant $E_\ep := E$ for all $\ep$. Then
  \begin{equation*}
    \sup_{t \in \R} \abs{ q_{\ep}(t) -  q(t) } \leq  \frac{C_E}{\sigma(E)} \ep,
  \end{equation*}
  where
  \begin{equation}\label{eqn:ivp-limit}
    q(t) = \frac{t}{\sigma(E)} + q_a \text{ for } t \in \R,
  \end{equation}
  and $C_E$ is defined in~\eqref{const}.
\end{theorem}

\begin{proof}
  Fix $t \in \R$. Then, using the fact that $t_{\ep}$ defined in~\eqref{eq:time-inverse} satisfies $t_{\ep}(q_{\ep}(t)) = t$ and
  the estimate~\eqref{test},
  \begin{multline*}
    \abs{ q_{\ep}(t) -  q(t)} = \abs{ q_{\ep}(t) -  \frac{t}{\sigma(E)} -q_a } 
    = \abs{ q_{\ep}(t) -  \frac{t_{\ep}(q_{\ep}(t))}{\sigma(E)}-q_a } \\ 
    = \frac{1}{\sigma(E)} \abs{\sigma(E) (q_{\ep}(t)-q_a) -  t_{\ep}(q_{\ep}(t)) } 
    \leq \frac{C_E}{\sigma(E)} \ep.
  \end{multline*}
  As this estimates hold for all $t \in \R$, the uniform convergence is established.
\end{proof}

The following lemma shows that the slopes of $q_{\ep}$ tend to $0$ as $E$ approaches the critical value $0 = \max V$.
	
\begin{lemma}
  \label{rate}
  For any $E>0$, let $\q(t)$ be given as in~\eqref{eqn:ivp-limit} in Theorem~\ref{mthm}.  Then it holds that
  $\q' \equiv \frac 1 {\sigma(E)}$ decreases to $0$ as $E$ decreases to $0 = \max V$.
\end{lemma}
	
\begin{proof}
  It suffices to show that $\sigma(E) \to \infty $ as $E \to 0$, where $\sigma(E)$ is defined in~\eqref{def:sigma-E}. Without loss of
  generality we assume that the maximum of $V(\ga)$ is attained at the origin $\ga=0$. Then 
  \begin{equation*}
    V(\ga) = \int_0^\ga\int_0^\tau V''(s)\d s\d\tau.
  \end{equation*}
  Now $V''$ is bounded as it is continuous and 1-periodic, so there exists a constant $C_V$ such that $V''>-C_V$ and hence
  \begin{equation*}
   V(\ga) \geq -\frac{1}{2}C_V \ga^2 \quad \text{ on } [0, 1].
  \end{equation*}
 Thus, for any $E>0$
  \begin{align*}
    \sigma(E) &= \int_0^1 \frac{1}{\sqrt{2(E-V(\ga))}} \d \ga \\
    &\geq \int_0^{1} \frac{1}{\sqrt{2E + C_V\ga^2}} \d \ga\\
    &=  \frac{1}{\sqrt{C_V}}\arcsinh\left( \sqrt{\frac{C_V}{2E}} \right).
  \end{align*}
  Taking $E\to 0$, it follows that $\sigma(E) \to \infty$, which establishes the claim.
\end{proof}

We recall that so far only the case $E>0$ was considered. For $E \leq 0$, the corresponding solutions of~\eqref{n2l} are either
separatrices or periodic trajectories with amplitudes of order $\ep$. This follows directly from the phase portrait, since the
spacing of equilibria (in direction of $q$) is $O(\ep)$. Therefore, in the limit $\ep \to 0$, solutions converge to constant
solutions (and hence zero slope). As a consequence, one can see from Lemma~\ref{rate} and Theorem~\ref{mthm} that the slope of the
limiting solution depends continuously on $E$.

We now consider the case where the initial data is fixed, $\q_{\ep}(0) = \q_a \neq 0$ and $\q'_{\ep}(0) = p_a > 0$, but the energy
$E_\ep$ is allowed to vary. We show this setting does not lead to a unique limit. Indeed, there is a family of possible limit
lines with different slopes; this can be achieved by exploiting the periodicity of $V$. So the homogenisation of the initial value
problem is, unlike the problem with fixed initial position and fixed energy, not well posed.

\begin{theorem}
  \label{theo:ivp-vel}
  Let $\q_\ep$ solve~\eqref{n2l} subject to $\q_{\ep}(0) = \q_a \neq 0$ and $\q'_{\ep}(0) = p_a > 0$. Then there exists a
  non-empty interval $I$, depending only on $\min V$ and $p_a$, such that for all $\alpha \in I$ there exists a sequence
  $\{\ep_k\}_{k\in\N}$ converging to zero such that the corresponding solutions $q_{\ep_k}$ converge uniformly to the line
  $q(t) = \alpha t + q_a.$ In particular, if $p_a < \sqrt{-2\min V}$, then $\inf I = 0$.
\end{theorem}
	
\begin{proof}
  Fix any $E >0$ such that $\min V \leq E - \frac{1}{2}p_a^2 \leq 0$.  Define $\ep_k$ for $k \in \mathbb N$ as
  \begin{equation*}
    \ep_k = \frac{q_a}{V_0^{-1} + k},
  \end{equation*}
  where $V_0^{-1}$ satisfies $V(V_0^{-1}) = E - \frac{1}{2}p_a^2$, which is possible for the given choice of $E$. This sequence
  $\{\ep_k\}_{k\in\N}$ obviously converges to 0, and the associated solutions $\q_{\ep_k}$ of the initial value problem all have
  fixed energy $E$.  Note that this utilises that a fixed $q_a \neq 0$ results in different values of the potential energy
  $V(q_a/\ep)$ by periodicity of $V$, as $\ep \to 0$. Thus the sequence $\left\{\q_{\ep_k}\right\}_{\ep_k}$ satisfies the
  assumptions of Theorem~\ref{mthm}.  Consequently $q_{\ep_k}$ converges to $q$ given in~\eqref{eqn:ivp-limit}. This can be
  repeated for any $E > 0$ provided $\min V \leq E - \frac{1}{2}p_a^2 \leq 0$. The set $I$ can thus be defined as
  \begin{equation*}
    I:= \left\{ \frac{1}{\sigma(E)} \bigm| E \in \left(\min \left\{\left(\tfrac{1}{2}p_a^2+\min V\right), 0 \right\}, 
    \tfrac{1}{2}p_a^2\right] \right\}
  \end{equation*}
  (so $E$ lies in a half-open set). If $p_a < \sqrt{-2\min V}$, then $\min \left\{\left(\tfrac{1}{2}p_a^2+\min V\right), 0
  \right\} = 0$, so we can consider the limit $E \to 0$. Since by Lemma~\ref{rate}
  \begin{equation*}
    \lim_{E \to 0} \frac{1}{\sigma(E)} = 0,
  \end{equation*}
  it follows that $\inf I = 0$ as claimed.
\end{proof}

In summary, the results of this section show that the homogenisation of the one-dimensional Hamiltonian problem with fixed initial
position and fixed total energy $E$ is meaningful, while the same problem with fixed initial position and momentum has a family of
solutions. The need to keep the energy $E$ fixed makes an averaging approach using the Maupertuis principle natural, and we will
present the result in Section~\ref{sec:Effect-one-dimension}.

\subsection{The one-dimensional boundary value problem}
\label{sec:bvp-1d}
	
In this section, we solve~\eqref{n2l} in one space dimension for $\q_{\ep}$ such that $\q_{\ep}(0) = \q_a$ and
$\q_{\ep}(T_{\ep}) = \q_b$. We consider the cases when either $T_{\ep} := T$ is fixed or $E_{\ep} := E$ is fixed for all $\ep > 0$
where, as before, $E$ and $E_{\ep}$ describe the energy defined by~\eqref{1d-energy}. Fixing $T_{\ep}$ requires that the
energy~\eqref{1d-energy} depends on $\ep$ in order to ensure that we reach the boundary value $q_\ep(T_\ep) = q_b$. Similarly,
fixing $E_{\ep}$ requires that the arrival time $T_{\ep}$ given by $T_{\ep} := t_{\ep}(q_b)$ depends on $\ep$. We first consider
the case of fixed energy.
	
\begin{theorem}
  \label{thm:bvp-E-energy}
  Fix $E > 0$ and suppose $q_{\ep}$ satisfies~\eqref{n2l} with $\q_{\ep}(0) = \q_a$, $\q_{\ep}(T_{\ep}) = \q_b$ and
  \begin{equation}
    \label{eq:bvp-E-energy}
    \frac{1}{2}\abs{\q_{\ep}'(t)}^2 + V(\q_{\ep}(t)/\ep) = E, \quad t \in \R. 
  \end{equation}
  Let $t_{\ep}$ be given by~\eqref{eq:time-inverse}, then $t_{\ep} \to \bar t$ uniformly as $\ep \to 0$, where
  \begin{equation}
    \label{eqn:bar-t}
    \bar t(q) = \sigma(E)(\q - \q_a),
  \end{equation}
  with $\sigma$ as defined in~\eqref{def:sigma-E}. Furthermore, $T_{\ep} \to \bar T$ as $\ep \to 0$, where
  \begin{equation}
    \label{eqn:t-bar}
    \bar T = \sigma(E)(\q_b - \q_a).
  \end{equation}
\end{theorem}

\begin{proof}
  Without loss of generality, we assume $q_b > q_a$. It follows from solving~\eqref{eq:bvp-E-energy} for $t=0$ that
  $p_{\ep} = \q_{\ep}'(0)$ satisfies
  \begin{equation*}
    p_{\ep} = \sqrt{2(E-V(\q_a/\ep))} 
  \end{equation*}
  (we recall $E > 0 = \max V$; the sign is taken to be positive since $q_b > q_a$). It follows from Lemma~\ref{tbd} that $t_{\ep}
  \to \bar t$ uniformly as $\ep \to 0$. In order to show that $q_{\ep}$ solves the boundary value problem, observe that
  $T_{\ep}=t_{\ep}(q_b)$ gives the unique arrival time. The fact that $T_{\ep} \to \bar T$ follows from the uniform convergence of
  $t_{\ep}$.
\end{proof}

\begin{corollary}
  \label{cor:bvp-E}
  Let $E > 0$ be fixed. The sequence $\q_{\ep}$ in Theorem~\ref{thm:bvp-E-energy} converges uniformly as $\ep \to 0$ to $q$ given
  in~\eqref{eqn:ivp-limit}. Furthermore $q(\bar T) = \q_b$.
\end{corollary}
        
\begin{proof}
  The first part follows from the fact that $t_{\ep} \to t$ as $\ep \to 0$ uniformly and the argument establishing
  Theorem~\ref{mthm}. The second part can be shown as follows. Let $\eta > 0$, then
  \begin{align*}
    \abs{q(\bar T) - \q_b} &\leq \abs{q(\bar T) - q_{\ep}(\bar T)} + \abs{q_{\ep}(\bar T) - q_{\ep}(T_{\ep})} + \abs{q_{\ep}(T_{\ep}) - q_b}.
  \end{align*}
  By assumption, $q_{\ep}(T_{\ep}) = q_b$, so the last term vanishes. Since $\q_{\ep}$ converges pointwise to $q$, there exists
  $\ep_0$ such that $\ep < \ep_0$ implies $\abs{q(\bar T) - q_{\ep}(\bar T)} < \eta/2$. As $\q_{\ep}$ converges to $q$ uniformly,
  the family $\{\q_{\ep}\}_\ep$ is equicontinuous by the `converse' direction in the Arzel\`a-Ascoli Theorem. Hence, by standard
  arguments we may take $\ep < \ep_1$ such that $\abs{q_{\ep}(T_{\ep}) - q_{\ep}(\bar T)} < \eta/2$. Consequently, for $\ep < \min
  \{ \ep_0, \ep_1\}$, it follows that $\abs{q(\bar T) - \q_b} < \eta$. Since $\eta$ was arbitrary, the result holds.
\end{proof}	
	
We now study the case of fixed arrival times $T_{\ep} := T$.

\begin{theorem}
  \label{thm:bvp-T-energy}
  Fix $T > 0$ and let $q_{\ep}$ satisfy~\eqref{n2l} with $\q_{\ep}(0) = \q_a$, $\q_{\ep}(T) = \q_b$ and
  \begin{equation}
    \label{eq:bvp-T-energy}
    \frac{1}{2}\abs{\q_{\ep}'(t)}^2 + V(\q_{\ep}(t)/\ep) = E_{\ep} \text{ for } t \in \R,
  \end{equation}
  for some $E_{\ep} > 0$. Let $t_{\ep}$ be defined by~\eqref{eq:time-inverse}. Then there exists a subsequence (not relabelled)
  $E_{\ep}$ and $\bar E > 0$ such that $t_{\ep} \to \bar t$ uniformly, where $\bar t$ is given by~\eqref{eqn:bar-t}, and $E_{\ep}
  \to \bar E$ as $\ep \to 0$.
\end{theorem}

\begin{proof}
  Without loss of generality, assume $\q_b > \q_a$. We seek $E_{\ep} > 0$ such that $t_{\ep}(\q_b) = T$, where $t_{\ep}$ is
  defined in~\eqref{eq:time-inverse}. We define $\tau_{\ep}$ as
  \begin{equation*}
    \tau_{\ep}(e) := \int_{\q_a}^{\q_b}  \frac{1}{\sqrt{2(e-V(s/\ep))}} \d s. 
  \end{equation*}
  Since $\max V = 0$, it follows that
  \begin{equation*}
    \int_{\q_a}^{\q_b}  \frac{1}{\sqrt{2(e-V(s/\ep))}} \d s \leq \int_{\q_a}^{\q_b}  \frac{1}{\sqrt{2e}} \d s ,
  \end{equation*}
  and thus $\tau_{\ep} \to 0$ as $e \to \infty$. Fixing $e > 0$, since $\tau_{\ep}(e) \to \sigma(e)(\q_b - \q_a)$ as $\ep \to 0$
  by Lemma~\ref{tbd}, there exists $\ep_0$ such that for $\ep < \ep_0$
  \begin{equation*}
    \sigma(e)(\q_b - \q_a) - 1 < \tau_{\ep}(e) . %% JZ < \sigma(e)(\q_b - \q_a) + 1.
  \end{equation*}
  By Lemma~\ref{rate} it follows that $\tau_{\ep} \to \infty$ as $e \to 0$. Thus, since $\tau_{\ep}$ is continuous, it follows by
  the intermediate value theorem that there is a $\tau_{\ep}(E_{\ep}) = T$, and for this choice of $E_{\ep}$, $t_{\ep}(\q_b) = T$.

  The next step is to show that the sequence ${E_{\ep}}$ is bounded. Suppose the contrary. Then there exists a subsequence (not
  relabelled) such that $E_{\ep} \to \infty$ as $\ep \to 0$. It follows that
  \begin{align*}
    T = \int_{\q_a}^{\q_b} \frac{1}{\sqrt{2(E_{\ep}-V(\gamma/\ep))}} \d \gamma 
        \leq \frac{\q_b - \q_a}{\sqrt{2E_{\ep}}} \to 0 
  \end{align*}
%% JZ $p_{\ep} = \q_{\ep}'(0)$ 
  as $\ep \to 0$. This contradicts the fact that $T$ is a positive constant. As $\{E_{\ep}\}$ is bounded, we may pass to a
  convergent subsequence (not relabelled) and denote its limit as $\bar E$. We now show that $\bar E >0$. It suffices to show that
  $\bar E \neq 0$. Suppose by contradiction that $E_{\ep} \to 0$ as $\ep \to 0$. By Lemma~\ref{tbd}, there exists $\ep_0$ such
  that for $\ep < \ep_0$
  \begin{equation*}
    \sigma(E_{\ep})(\q_b - \q_a) -1 < T = t_{\ep}(q_b). %% JZ  < \sigma(E_{\ep})(\q_b - \q_a) +1. 
  \end{equation*}
  Hence $T \to \infty$, since $\sigma(E_{\ep}) \to \infty$ as $E_{\ep} \to 0$ by Lemma~\ref{rate}, contradicting the finiteness of
  $T$. Thus $\bar E > 0$, and using Lemma~\ref{tbd} it follows for the subsequence $E_{\ep} \to \bar E > 0$ as $\ep \to 0$ that
  \begin{equation*}
    \abs{t_{\ep}(q) - \sigma(E_{\ep})(\q - \q_a)} \leq \frac{2\ep}{\sqrt{2E_{\ep}}} 
    < \frac{2\ep}{\sqrt{ \bar E}}
  \end{equation*}
  for $\ep < \ep_0$, where $\ep_0$ is chosen such that
  \begin{equation}
    \label{eq:bounds}
    \abs{E_{\ep} - \bar E}  < \bar E/2.
  \end{equation}
  It follows that $t_\ep \to t$ uniformly as $\ep \to 0$.
\end{proof}

\begin{corollary}
  Fix $T > 0$. Then any convergent subsequence $\q_{\ep}$ as in Theorem~\ref{thm:bvp-T-energy}, with $E_{\ep} \to \bar E$,
  converges uniformly as $\ep \to 0$ to $q$ defined in~\eqref{eqn:ivp-limit}.
\end{corollary}

\begin{proof}
  We can argue along the lines of the proof of Corollary~\ref{cor:bvp-E}, noting that $T_{\ep} = T$ is fixed. To obtain a uniform
  bound observe that there exists $\ep_0$ such that for $\ep < \ep_0$ the inequality in~\eqref{eq:bounds} hold. Then we can apply
  the same line of reasoning as in the proof of Theorem~\ref{thm:bvp-T-energy}.
\end{proof}	

The results in this section are analogous to those for the initial value problem, in the sense that the averaging problem with
fixed total energy $E$ gives a well-defined limit. This suggests that averaging via the Maupertuis principle is promising, and
indeed we will show this in the next section. One can see that, analogously to the initial-value problem studied in
Section~\ref{sec:ivp-1d}, the second natural problem, namely now with a fixed final time $T$, can have multiple subsequences with
different limits, as in the setting of Theorem~\ref{theo:ivp-vel}.

\section{Effective one-dimensional dynamics: Homogenisation techniques}
\label{sec:Effect-one-dimension}

Theorem~\ref{mthm} gives a complete description of the solution trajectories for the initial value problem~\eqref{n2l} in the
limit $\ep \to 0$. We now show that both the proposed averaging procedure via the Maupertuis principle and the established
homogenisation technique via Hamilton-Jacobi theory recover this result. An analogous statement holds for the boundary value
problem with fixed energy, where the effective limit is given in Theorem~\ref{thm:bvp-E-energy}. In particular, the results of
this section show that the averaging approaches via Maupertuis and Hamilton agree in one space dimension.

\subsection{Effective one-dimensional dynamics via Maupertuis}
\label{sec:Effect-one-dimension-1}

The proposed averaging approach via Maupertuis recovers the result of Theorem~\ref{mthm} in a trivial manner. Indeed, the geodesic
connecting two given points is the straight line segment joining these points, hence the minimisation of the Jacobi functional
immediately yields a straight line, as~\eqref{eqn:ivp-limit} in Theorem~\ref{mthm}. All nontrivial information is contained in the
time reparametrisation. The parametrisation~\eqref{eq:maupert-time} is the limit given in~\eqref{eqn:bar-t}
with~\eqref{def:sigma-E}.

\subsection{Effective one-dimensional dynamics via Hamilton-Jacobi}
\label{sec:Effect-one-dimension-2}

We now show that the traditional averaging approach via Hamilton-Jacobi theory recovers the same effective limit as the explicit
computation (Theorem~\ref{mthm}) and Maupertuis averaging (Section~\ref{sec:Effect-one-dimension-1}). In the Hamilton-Jacobi
approach, the Hamiltonian dynamics is reformulated equivalently as Hamilton-Jacobi equation
\begin{equation}
  \label{eqn:hj-micro}
  \partial_t u_{\ep}(x,t) = \mathscr H\left(\frac{x}{\ep}, \partial_x u_{\ep}(x,t) \right),
\end{equation}
where $\partial_t$ and $\partial_x$ denote the partial derivatives with respect to $t$ and $x$. The Hamiltonian $\mathscr H$ is given by
\begin{equation*}
  \mathscr H(q,p) = \frac{1}{2}\abs{p}^2 + V\left( q \right).
\end{equation*}
A key result of~\cite{Lions1987a} is that the viscosity solutions $u_{\ep}$ of~\eqref{eqn:hj-micro} converge, as $\ep \to 0$, to a
solution $u$ of a Hamilton-Jacobi equation of the form
\begin{equation}
  \label{hjpde1}
  \partial_t u(x,t)= \bar {\mathscr H}\left(\partial_x u(x,t) \right).
\end{equation}
For the one-dimensional problem considered here, the effective Hamiltonian $\bar {\mathscr H}$ has been calculated
in~\cite{Lions1987a, E1991a, Gomes2000a,Evans2001a,Evans2002a} and is defined as
\begin{equation}
  \label{hbar}
  \bar {\mathscr H}(p) :=
  \begin{cases}
    0 & \text{ if } \abs{p} \leq \displaystyle \int_0^1 \sqrt{-2V(\ga)} \d \ga,\\
    \alpha & \text{ if }  \abs{p} =\displaystyle \int_0^1 \sqrt{2(\alpha-V(\ga))} \d \ga.
  \end{cases}
\end{equation}
	
The ``flat piece'' occurring in the in the Hamiltonian, due to the lack of differentiability in the corresponding Lagrangian, has
been described as ``trapping'' (see the discussion in~\cite[Section 7]{E1991a}). In order to solve the Hamilton-Jacobi
PDE~\eqref{hjpde1}, one uses the ansatz $u(x,t) = v(x) + w(t)$. Substituting this into~\eqref{hjpde1} it follows that
\begin{align*}
  w'(t) &= E,\\
  \bar {\mathscr H}\left(v'(x) \right) &= E,
\end{align*}
for some $E \in \mathbb R$. Reserving the free constant of integration for the spatial part $v(x)$, we can fix $w(t) =
Et$. By~\eqref{hbar} for any $E>0$ one has that either
\begin{equation*}
  v'(x) = \int_0^1 \sqrt{2(E-V(\ga))} \d \ga  \quad 
  \text{ or } \quad  v'(x) = -\int_0^1 \sqrt{2(E-V(\ga))} \d \ga.
\end{equation*}
Consequently  a solution of~\eqref{hjpde1} is
\begin{equation}
  u_{\pm}(x,t) = \pm \left( \int_0^1 \sqrt{2(E-V(\ga))} \d \ga \right) x + Et + C.
\end{equation}
for some $C \in \R$. We recall Jacobi's Theorem, paraphrased for this context, from~\cite[Section 3.2, Theorem 1]{Evans1998a},
which connects the function $u$ to trajectories of~\eqref{n2l}.

\begin{theorem} 
  Let $u \in C^2(\R)$ solve~\eqref{hjpde1}. Assume that $q$ satisfies the ODE 
  \begin{equation}
    \label{impo} 
    q' = \bar {\mathscr H}'(p), 
  \end{equation} 
  where $p = \partial_x u(q)$; the slight abuse of notation introduced by re-using $q$ and ${\mathscr H}$ is motivated by the fact
  that $q$ and $p$ are solutions to the characteristic equations for~\eqref{hjpde1}, namely
  \begin{align*}
    q' &= \bar {\mathscr H}'(p),\\
    p' &= 0.
  \end{align*}
\end{theorem}

The inverse function theorem, assuming that
\begin{equation*}
  p > \int_0^1 \sqrt{-2V(\ga)} \d \ga,
\end{equation*}
gives that
\begin{equation*}
  \bar {\mathscr H}'(p) = \frac{1}{p'\left(\bar {\mathscr H}(p)\right)} , 
\end{equation*}
where
\begin{equation}
  \label{eq:p}
  p(\alpha) := \int_0^1 \sqrt{2(\alpha-V(\ga))} \d \ga, \text{ for } \alpha > 0.
\end{equation}
Now $\bar {\mathscr H}(p) = E$ and therefore 
\begin{equation*}
  \label{impo2}
  q' = \frac{1}{p'\left(\bar {\mathscr H}(p)\right)} =\frac{1}{p'\left(E\right)}.
\end{equation*}
An integration, given the initial position $q_a$ at $t=0$, results in~\eqref{eqn:ivp-limit}. The fact that $\sigma(E) =
p'\left(E\right)$ can be seen by differentiating $p$ defined in~\eqref{eq:p} and a comparison with~\eqref{def:sigma-E}. We have
thus shown that the solutions obtained by averaging the Hamilton-Jacobi equation coincide with those obtained by the direct
homogenisation of Hamilton's ODEs.

Having established the equivalence of explicit computation, Maupertuis averaging and Hamilton-Jacobi averaging for fixed energy
$E$, we close the section on Hamilton-Jacobi theory by remarking that in the situation studied, $\bar {\mathscr H} \in C^1(\R)$.
Indeed, we have just shown that $\bar {\mathscr H}'(p) = \frac{1}{\sigma(E)} = \frac{1}{\sigma(\bar {\mathscr H}(p))} \to 0$ by
Lemma~\ref{rate} as $p \to \int_0^1 \sqrt{-2V(\ga)} \d \ga$.  Since $\bar {\mathscr H}'(p) >0$ for
$p > \int_0^1 \sqrt{-2V(\ga)} \d \ga$, it also holds that $\bar {\mathscr H}(p) \to 0$ as $p \to \int_0^1 \sqrt{-2V(\ga)} \d \ga$
by~\eqref{hbar}; therefore $\bar {\mathscr H} \in C^1(\R)$.

\section{Higher dimensional effective dynamics}
\label{sec:nd}
	
In this section, we consider the effective dynamics of
\begin{equation}
  \label{n2lnd2}
  \ddot q_{\ep} = -\frac{1}{\ep} \,  \nabla V\left(\frac{q_{\ep}}{\ep}\right),\; \ep>0,
\end{equation}
where we choose $q(0)=q_a$ with $E > 0$ fixed. The standing assumptions on $V$ are that
$V \in L^{\infty}(\R^\dim) \, \cap \, C^{2}(\R^\dim)$, that $V$ is $[0,1]^\dim$-periodic and that $\max V = 0$. One can consider
this problem using the Hamilton-Jacobi PDE as discussed in~\cite{Lions1987a}. In this section, we describe an alternative based on
the Maupertuis principle. Recall that the \emph{Jacobi length functional}, whose critical points are reparametrisation of the
solutions to~\eqref{n2lnd2}, is given by~\eqref{length},
\begin{equation}
\label{length2}
  \jac_E^{\ep}[\ga] := \int_0^1 2(E-V(\ga(s)/\ep)) \norm{\ga'(s)}^2 \d s \text{ for } \ga \colon [0,1] \to \mathbb R^\dim.
\end{equation}
 Since the Maupertuis formulation is a geometric one, it turns out that averaging results are readily available
in the case considered here, namely $E > 0 = \max V$. We collect the results below.

As before, we are interested in the convergence of solutions of~\eqref{n2lnd2} as $\ep \to 0$. A tool to study the convergence of
minimisers is $\Gamma-$convergence (see for example~\cite{Braides2002a, Dal-Maso1993a}).
\begin{definition}[$\Gamma$-convergence]
  Let $F_{\ep} \colon X \to [0,\infty]$, $F \colon X \to [0,\infty]$ be functionals defined on a metric space $(X,d)$. Then
  $F_{\ep} \stackrel{\Gamma}{\longrightarrow} F$ if for every $u \in X$ and any sequence $\{ u_{\ep}\}_{\ep >0} \subset X$ such
  that $u_{\ep} \stackrel{d}{\to} u$, one has
  \begin{equation*}
    \liminf_{\ep \to 0} F_{\ep}(u_{\ep}) \geq F(u),
  \end{equation*}
  and if in addition for every $u \in X$ there exists a sequence $\{ u_{\ep} \}_{\ep >0} \subset X$ such that
  $u_{\ep} \stackrel{d}{\to} u$ with
  \begin{equation*}
    \limsup_{\ep \to 0} F_{\ep}(u_{\ep}) \leq F(u),
  \end{equation*}
  where $\stackrel{d}{\to}$ denotes convergence in the given metric. To indicate the underlying topology, we write $\Gamma(L^2)$
  for $\Gamma$-convergence in $L^2$.
\end{definition}

The notion of $\Gamma$-convergence gives, along with suitable coercivity estimates, sufficient conditions to ensure that
minimisers converge to minimisers of the limiting functional. The next result studies the $\Gamma$-convergence of Jacobi
functionals and gives a characterisation of the limit. In the case we consider, it follows that~\eqref{length2} is the energy
functional of a Riemannian metric, whose $\Gamma-$convergence has been previously determined~\cite[Chapter 16]{Braides1998a}. We
re-state the result for the functional~\eqref{length2}.

\begin{theorem}
  \label{brd}
  Fix $E>0$. There exists a non-negative convex function $\bar \jac \colon \R^\dim \to \R$ such that
  \begin{equation*}
    \Gamma(L^2)-\lim_{\ep \to 0} \jac_E^{\ep}[\ga]
    = \int_0^1 \bar \jac( \gamma'(s)) \d s
  \end{equation*}
  for $\gamma \in H^1(0,1)$. Furthermore, the function $\bar \jac$ satisfies
  \begin{equation*}
    \bar \jac(z) := \lim_{\ep \to 0} 
    \inf \left\{ \jac_E^{\ep}[\ga] \bigm| \gamma \in H^1(0,1), \gamma(0)=0, \gamma(1)=z \right\}.
  \end{equation*}
\end{theorem}

Restricting the class of curves joining two specified points in $\R^\dim$ does not affect the $\Gamma$-convergence result for the
Jacobi energy or length functional; to see this, one can apply~\cite[Proposition 11.7]{Braides1998a}. It is natural to ask what
properties the effective energy density $\bar \jac$ possesses, and whether the formula for $\bar \jac$ can be evaluated for
specific potentials. We first note that the function $\bar \jac$ is in general not a quadratic form, see~\cite[Chapter
16]{Braides1998a} for a counter-example. Instead, $\bar \jac$ describes a Finsler metric, that is,
\begin{equation*}
  2E\norm{q}^2  \leq \bar \jac(q) \leq 2\left(E-\min V \right)\norm{q}^2, 
\end{equation*}
and $\bar \jac$ is convex and 2-homogeneous (see~\cite{Braides2002d}). Consequently, the $\Gamma$-con\-ver\-gence of the Jacobi
energy functional gives rise to a well-posed limit problem. By equicoercivity of the Jacobi energy functionals, one can apply the
fundamental theorem of $\Gamma$-convergence~\cite[Theorem 7.2]{Braides1998a} to find that each sequence of minimisers has a
convergent subsequence, whose limit is a geodesic of the limiting Finsler metric. Additionally, from~\cite{Braides1992b} we have
the following result for the Hamilton principle.

\begin{theorem}
  \label{brd2}
  There exists a non-negative convex function $\bar {\mathscr L} \colon \R^\dim \to \R$ such that
  \begin{equation*}
    \Gamma(L^2)-\lim_{\ep \to 0} \act_T^{\ep}[\q] 
    = \int_0^T \bar {\mathscr L}( \q'(t)) \d t
  \end{equation*}
  for $\q \in H^1(0,T)$. Furthermore, the function $\bar {\mathscr L}$ satisfies
  \begin{equation*}
    \bar {\mathscr L}(z) := \lim_{\ep \to 0} 
    \inf \left\{ \act_T^{\ep}[\q] \bigm|  \q \in H^1(0,T), \q(0)=0, \q(T)=z \right\}.
  \end{equation*}
\end{theorem}
As with the Jacobi functional, the result continues to hold when restricted to the boundary value problem (compare
Section~\ref{sec:bvp-1d}). In Theorems~\ref{brd} and~\ref{brd2} we see that the $\Gamma$-limit is expressed in terms of the
minimum values of boundary value problems.

\subsection{Relation between critical points of Maupertuis and Hamilton}
\label{sec:Relat-betw-crit}

In order to explore the relationship between $\bar \jac$ and $\bar {\mathscr L}$ further, we need to develop the relationship
between minimisers of $\jac_E^{\ep}$ and $\act_T^{\ep}$. As a first step, we demonstrate that suitable minimisers of the action
functional also minimise of the Jacobi functional; it will then easily follow that the same holds for saddle points. For
notational convenience, we assume that $\q_a, \q_b \in \R^\dim$ are fixed throughout this section. In addition, we define
\begin{align*}
  \Gamma &:= \set{\ga \in H^1(0,1)\bigm| \ga(0) = \q_a, \ga(1) = \q_b},\\
  G_T &:= \set{\q \in H^1(0,T) \bigm| \q(0) = \q_a, \q(T) = \q_b}.
\end{align*}
While there is only a minor difference between the spaces $\Gamma$ and $G_T$, we keep them notationally different to emphasise the
purpose of the curves in each space. Curves in $\Gamma$ are thought of as admissible geodesics on which we evaluate the Jacobi
functional, parametrised by arc-length. Curves in $G_T$ are thought of as admissible trajectories of the action functional,
parametrised by time.

\begin{theorem}
  \label{T.1}
  Fix $T > 0$ and $\ep > 0$ and let $V \in C^2(\R^\dim) \cap L^{\infty}(\R^\dim)$ with $\max V = 0$. Suppose that $\bar \q$
  minimises $\act_T^{\ep}$ in $G_T$. Let $E_T \in \mathbb R$ be such that
  \begin{equation}
    \label{energyisE}
    E_T := \frac{1}{2}\norm{\bar \q'(t)}^2 + V(\bar \q(t)/\ep), \quad t \in [0,T].
  \end{equation}
  There exists a diffeomorphism $\psi \colon [0,1] \to [0, T]$ such that $\bar \ga:= \bar \q \, \circ \, \psi$ satisfies
  \begin{equation*}
    \jac_{E_T}^{\ep}[\bar \ga] = \min \jac_{E_T}^{\ep}[\ga],
  \end{equation*}
  where the minimisation is over $\ga \in \Gamma$ such that
  \begin{equation}
    \label{timeisT}
    \int_0^1 \frac{\norm{\ga'(s)}}{\sqrt{2(E_T-V(\ga(s)/\ep))}}\; \textup{d} s = T.
  \end{equation}
  Conversely, fix $E > 0$ and $\ep > 0$. Suppose that $\bar \ga$ minimises $\jac_E^{\ep}$ in $\Gamma$. Define $T_E > 0$ as
  \begin{equation}
    \label{time-constant}
    T_E:= \int_0^1 \frac{\norm{\bar \ga'(s)}}{\sqrt{2(E-V(\bar \ga(s)/\ep))}} \d s.
  \end{equation}
  Then there exists a diffeomorphism $\psi \colon [0,1] \to [0, T_E]$ such that
  $\bar \q := \bar \ga \, \circ \, \psi^{-1}$ satisfies
  \begin{equation*}
    \act_{T_E}^{\ep}[\bar \q] = \min \act_{T_E}^{\ep}[\q],
  \end{equation*}
  where the minimisation is over $\q \in G_{T_E}$ such that 
  \begin{equation}
    \label{energyisE2}
    \frac{1}{2}\norm{\q'(t)}^2 + V(\q(t)/\ep) = E.
  \end{equation}
\end{theorem}

\begin{proof}
  We recall the definition of the length functional associated with $\jac_E$,
  \begin{equation*}
    \nac_E^{\ep}[\ga] :=  \int_0^1 \sqrt{2(E_T-V(\ga(s)/\ep))} \norm{\ga'(s)} \d s.
  \end{equation*}
  It follows from classical results~\cite[Lemmas 1.4.2, 1.4.5]{Jost2005a} that
  \begin{equation}
    \label{eq:leeq}
    \jac_E^{\ep}[\ga] \geq \frac{1}{2}\nac_E^{\ep}[\ga]^2,
  \end{equation}
  with equality when $\ga$ is parametrised proportional to arc-length. Minimisers of the energy functional $\jac_E^{\ep}$ are
  always parametrised by arc-length, whereas $\nac_E^{\ep}$ is invariant under reparametrisation.
		
  To prove the first part of the result, fix $T > 0$. Let $\psi \colon [0,1] \to [0, T]$ be any diffeomorphism, and set
  $\bar \ga:= \bar \q \, \circ \, \psi$. Then
  \begin{align*}
    \nac_{E_T}^{\ep}[\bar \ga] &= \int_0^1 \sqrt{2(E_T-V(\bar \ga(s)/\ep))} \norm{\bar \ga'(s)} \d s\\
                               &= \int_0^1 \sqrt{2(E_T-V(\bar \q(\psi(s)/\ep)))} \norm{\bar \q'(\psi(s))} \abs{\psi'(s)} \d s\\
                               &= \int_0^T \sqrt{2(E_T-V(\bar \q(t)/\ep))} \norm{\bar \q'(t)} \d t.
  \end{align*}
  It follows from~\eqref{energyisE} that
  \begin{align} 
    \nac_{E_T}^{\ep}[\bar \ga] &= \int_0^T \norm{\bar \q'(t)}^2 \d t \nonumber \\
                               &=  \int_0^T \left[ \frac{1}{2}\norm{\bar \q'(t)}^2 - V(\bar \q(t)/\ep) \right] \d t 
                                 +  \int_0^T \left[ \frac{1}{2}\norm{\bar \q'(t)}^2 + V(\bar \q(t)/\ep) \right] \d t \nonumber \\
                               &=  \act_T^{\ep}[\bar \q] + E_TT. \label{2.1}
  \end{align}
  Since $\bar q$ minimises $\act_T^{\ep}$ in $G_T$, we conclude that
  \begin{equation}
    \label{1.5}
    \nac_{E_T}^{\ep}[\bar \ga] \leq \act_T^{\ep}[\q] + E_TT
  \end{equation}
  for any $\q \in G_T$ such that~\eqref{energyisE} holds. Let $\ga \in \Gamma$ and suppose further that~\eqref{timeisT} holds and
  $\norm{\gamma'} > 0$; the latter can be achieved by a reparametrisation if necessary. We seek a diffeomorphism $\psi \colon
  [0,1] \to [0,T]$ such that~\eqref{energyisE} holds for $\q:= \ga \, \circ \, \psi^{-1}$. Note that~\eqref{energyisE} holds for
  $\q$ if and only if
  \begin{equation}
    \label{1.3}
    \frac{1}{2} \frac{1}{\abs{\psi'(s)}^2}\norm{ \ga'(s)}^2 + V(\ga(s)/\ep) = E_T, \quad \forall s \in [0,1].
  \end{equation}
  Taking $\psi' > 0$, without loss of generality, we solve~\eqref{1.3} to obtain
  \begin{equation}
    \label{1.4}
    \psi(s) := \int_0^s \frac{\norm{\ga'(s)}}{\sqrt{2(E_T-V(\ga(s)/\ep))}} \d s,
  \end{equation}
  which is well defined since~\eqref{timeisT} holds for $\ga$. Furthermore, since $\norm{\ga'} > 0$, it follows that $\psi$ is a
  diffeomorphism. Consequently we have
  \begin{align}
    \nac_{E_T}^{\ep}[\ga] &= \int_0^1 \sqrt{2(E_T-V(\q(\psi(s)/\ep)))} \norm{\q'(\psi(s))} \abs{\psi'(s)} \d s \nonumber\\
                          &=  \int_0^T \norm{\q'(t)}^2 \d t = \act_T^{\ep}[\q] + E_TT, \label{1.6}
  \end{align}
  where we have used~\eqref{energyisE} for the second identity. Combining~\eqref{1.5},~\eqref{1.6} and observing that $\ga$ is
  arbitrary gives that
  \begin{equation}
    \label{jac-min-1}
    \nac_{E_T}^{\ep}[\bar \ga] \leq \nac_{E_T}^{\ep}[\ga]
  \end{equation}
  for all $\ga \in \Gamma$ such that $\norm{\ga '} > 0$. Furthermore, since $\max V = 0$, it follows by~\eqref{energyisE} that
  $E_T > 0$, therefore $\nac_{E_T}^{\ep}$ is the energy functional corresponding to a Riemannian length functional. It is well
  known~\cite{Jost2005a} that minimisers are regular curves and therefore~\eqref{jac-min-1} holds for all $\ga \in \Gamma$.
  Finally, observing that~\eqref{eq:leeq} holds proves the first part of the result, since $\nac_{E_T}^{\ep}$ is invariant under
  reparametrisations; we choose the parametrisation by arc-length.
		
  For the second part of the result we use~\eqref{eq:leeq} to see that $\bar \gamma$ is a minimiser of $\nac_E^{\ep}$. Observe
  that, with $\psi$ as in~\eqref{1.4}, fixed $t \in [0, T_E]$ and $s$ such that $t = \psi(s)$,
  \begin{align}
    \frac{1}{2}\norm{\bar \q'(t)}^2 + V(\bar \q(t)/\ep) 
    = \frac{1}{2} 
      \left(\psi'(s) \right)^{-2}\norm{\bar \ga'(s)}^2 + V(\bar \ga(s)/\ep) 
     = E, \label{3.1}
  \end{align}
  hence~\eqref{energyisE2} holds. Consequently, using the change of variables $t = \psi(s)$ and~\eqref{3.1},
  \begin{multline*}
    \act_{T_E}^{\ep}[\bar \q] + ET_E = \int_0^{T_E} \norm{\q'(t)}^2 \d t 
                                     = \int_0^1  \norm{\bar \q'(\psi(s))}^2 \abs{\psi'(s)} \d s \\
                                     = \int_0^1 \sqrt{2(E-V(\bar \ga(s)/\ep))} \norm{ \ga'(s) } \d s 
                                      = \nac_E^{\ep}[\bar \ga].
  \end{multline*}
  Since $\bar \ga$ minimises $\jac_E^{\ep}$ in $\Gamma$, it holds that
  \begin{equation}\label{3.3}
    \act_{T_E}^{\ep}[\bar \q] + ET_E \leq \nac_E^{\ep}[\ga]
  \end{equation}
  for all $\ga$ such that $\ga \in \Gamma$.  Let $\q \in G_{T_E}$ satisfy~\eqref{energyisE2} and define
  $\ga := \q \, \circ \, \psi$. Then
  \begin{align}
    \nac_E^{\ep}[\ga] &= \int_0^1 \sqrt{2(E-V(\ga(s)/\ep))} \norm{\ga'(s)} \d s \nonumber \\
                      &= \int_0^1 \sqrt{2(E-V(\q(\psi(s)/\ep)))} \norm{\ga'(\psi(s))} \abs{ \psi'(s)} \d s \nonumber \\
                      &= \int_0^{T_E} \norm{\q'(t)}^2 \d t \nonumber \\
                      &=   \act_{T_E}^{\ep}[\q] + ET_E, \label{3.4}
  \end{align}
  where we have used the fact that $\q$ satisfies~\eqref{energyisE2}. Combining~\eqref{3.3} and~\eqref{3.4} gives the second part
  of the result.
\end{proof}

There are no maximisers of the functionals considered, since the functionals are unbounded from above. Thus, it is natural to
explore in light of Theorem~\ref{T.1} and Corollary~\ref{C.1} whether an analogous result holds also for saddle points. The answer
is affirmative.

\begin{corollary}
  \label{C.1}
  Fix $\ep > 0$. Let $V$ be as in Theorem~\ref{T.1}. Fix $T > 0$ and suppose that $\bar \q \in G_T$ is a saddle point of
  $\act_T^{\ep}$ such that~\eqref{energyisE} holds for some $E_T \in \mathbb R$. Taking $\psi \colon [0,1] \to [0, T]$ as
  in~\eqref{1.4}, it follows that $\bar \ga := \bar\q\, \circ \, \psi$ is a saddle point of $\jac_{E_T}^{\ep}$.  Conversely, fix
  $E > 0$ and suppose that $\bar \ga \in \Gamma$ is a saddle point of $\jac_E^{\ep}$ satisfying~\eqref{time-constant} for some
  $T_E \in \mathbb R$. Then there exists a diffeomorphism $\phi \colon [0,1] \to [0, T_E]$ such that
  $\bar \q := \bar \ga \, \circ \, \phi^{-1}$ is a saddle point of $\act_{T_E}^{\ep}$.
\end{corollary}

\begin{proof}
  Fix $T > 0$ and consider a saddle point of $\act_T^{\ep}$. By the Maupertuis principle~\cite{Arnold1989a, Marsden1999a},
  $\bar \ga$ is a critical point of $\jac_{E_T}^{\ep}$. It is not possible that $\bar \ga$ is a minimiser of $\jac_{E_T}^{\ep}$,
  since by Theorem~\ref{T.1} it would be a minimiser $\act_T^{\ep}$. Also, it cannot be a maximiser, since $\jac_{E_T}^{\ep}$
  possesses no maximisers. To see this, note that any curve on a Riemannian manifold can be perturbed to increase length. The
  second part of the result follows by similar reasoning.
\end{proof}

It is natural to ask, given the knowledge of how minimisers between the Maupertuis and Hamilton principles are related, if there
is a similar relationship between their $\Gamma$-limits $\bar \jac$ and $\bar \act$. There seems to be no immediate
relationship. In case where $T$ is fixed, there is a functional relationship between the minimum values of $\jac_{E_T}^{\ep}$ and
$\act_T^{\ep}$. However, this minimum value of $\jac_{E_T}^{\ep}$ is attained subject to a constraint on length. In general, the
minimum length does not meet this constraint. Therefore, while a functional relationship exists at the $\ep$-scale, it will not
hold for the $\Gamma$-limit. A similar line of reasoning also holds for the converse problem.

\section*{Acknowledgements}
The authors are grateful for funding of an EPSRC network grant (EP/F03685X/1), which stimulated the research described here. DCS
and JZ thank the EPSRC for funding the project EP/K027743/1. JZ also received funding through the Leverhulme Trust (RPG-2013-261)
and a Royal Society Wolfson Research Merit Award. All authors thank the reviewers for helpful suggestions.
	
%\bibliographystyle{alpha}
%\bibliography{jz}

\def\cprime{$'$} \def\cprime{$'$} \def\cprime{$'$}
  \def\polhk#1{\setbox0=\hbox{#1}{\ooalign{\hidewidth
  \lower1.5ex\hbox{`}\hidewidth\crcr\unhbox0}}} \def\cprime{$'$}
  \def\cprime{$'$}

\end{document}